\documentclass[a4paper,12pt]{article}
\usepackage{amsmath}
\usepackage{amssymb}
\usepackage{amsthm} 
\usepackage{booktabs}

\usepackage{tikz} 
\usetikzlibrary{arrows.meta,positioning}

\theoremstyle{plain}
\newtheorem{theorem}{Theorem}
\newtheorem{proposition}[theorem]{Proposition}

\newtheorem{lemma}[theorem]{Lemma}
\newtheorem{claim}[theorem]{Claim}
\newtheorem{example}[theorem]{Example}

\theoremstyle{definition}
\newtheorem{definition}[theorem]{Definition}

\date{July 30, 2026}
\author{Hiroki Kodama}
\title{
Structural Oscillatority Criterion of Boolean Networks 
}

\begin{document}

\maketitle

\begin{abstract}
Azuma et al.\ showed that a cactus-expandable Boolean network $\Sigma$ is structurally oscillatory if all the simple cycles of $\Sigma$ contain an odd number of inhibiting edges.
We show that we do not need the cactus-expandable condition. Namely, a strongly connected Boolean network $\Sigma$ is structurally oscillatory if and only if all the simple cycles of $\Sigma$ contain an odd number of inhibiting edges.
Additionally, we provide a characterization of a structurally oscillatory Boolean network for the general case.
\end{abstract}

\section{Introduction and Notation}

\subsection{Boolean dynamics}

We call $\{0, 1\}$ a Boolean domain, where $0$ and $1$ correspond to \textit{false} and \textit{true} respectively.
We define operators $\bar{\quad}$, $\vee$, and $\wedge$, that correspond to \textit{NOT}, \textit{OR}, and \textit{AND} respectively.
$$
\bar{0} =1, \bar{1}=0; 
$$
$$
0 \vee 0 = 0, 0 \vee 1 = 1 \vee 0 = 1 \vee 1 = 1; 
$$
$$
0 \wedge 0 = 0 \wedge 1 = 1 \wedge 0 = 0, 1 \wedge 1 = 1.
$$

We also endow $\{0, 1\}$ with an order $0<1$.

Fix a positive integer $N$ and consider a function $\varphi \colon \{0, 1\}^N \to \{0, 1\}, (x_1,\dots,x_N) \mapsto \varphi(x_1,\dots,x_N)$.

$\varphi$ is said to be independent of $x_i$, weakly increasing in $x_i$, or weakly decreasing in $x_i$ respectively if 
$$\varphi(x_1,\dots,x_{i-1},0,x_{i+1},\dots,x_N)=    \varphi(x_1,\dots,x_{i-1},1,x_{i+1},\dots,x_N),$$
$$\varphi(x_1,\dots,x_{i-1},0,x_{i+1},\dots,x_N)\leq \varphi(x_1,\dots,x_{i-1},1,x_{i+1},\dots,x_N),$$
$$\varphi(x_1,\dots,x_{i-1},0,x_{i+1},\dots,x_N)\geq \varphi(x_1,\dots,x_{i-1},1,x_{i+1},\dots,x_N)$$
for any $(x_1,\dots,x_{i-1},x_{i+1},\dots,x_N)\in\{0, 1\}^{N-1}$ respectively.

$\varphi$ is said to be increasing in $x_i$, or decreasing in $x_i$ respectively if 
$\varphi$ is not independent of $x_i$ and is weakly increasing in $x_i$, or weakly decreasing in $x_i$ respectively.

We say that $\varphi$ is monotone if for any $i \in \{1,\dots,N\}$, $\varphi$ is independent of $x_i$, increasing in $x_i$, or decreasing in $x_i$ .

\bigskip

Suppose $f \colon \{0, 1\}^N \to \{0, 1\}^N$ is a map. Then $f$ defines a discrete dynamical system on 
$\{0, 1\}^N$ by $x(t+1):=f(x(t))$. In this paper, we only consider the case that for any $j \in \{1,\dots,N\}$,
$f_j=\pi_j\circ f$ is monotone, where $\pi_j\colon \{0, 1\}^N \to \{0, 1\}$ is the $j$-th projection.
We call such a dynamical system a Boolean dynamical system.

We are now going to visualize Boolean dynamical systems by digraphs. Before that, we review the digraphs.

\subsection{Digraphs}

A digraph, or directed graph, is an ordered pair $ G = (V,A)$,
where $V$ is a finite set of vertices and $A$ is a finite set of ordered pairs (called arcs) of vertices of $V$, in other words, $A\subset V \times V$.
In this paper, we allow loops, i.e.\ $(v,v)$ can be an arc.

For an arc $a=(v,w)$, we use the notation that $s(a)=v$ and $t(a)=w$. We also say that
$a$ is an arc from $v$ to $w$.
For a vertex $v \in V$, 
the indegree $\delta^-(v)$ is the number of arcs to $v$ and
the outdegree $\delta^+(v)$ is the number of arcs from $v$.

A path of length $n \geq 1$ from $v$ to $w$ is a sequence of $n+1$ vertices $(v_0,v_1,\dots,v_{n-1},v_n)$ such that
$v_0=v$, $v_n=w$ and $(v_{i-1},v_i)\in A$ for $1 \leq i \leq n$.
A path is called a simple path if $v_0, \dots, v_{n-1}, v_n$ are distinct.

A cycle is a path with $v_0 = v_n$.
A cycle is called a simple cycle if $v_0, \dots, v_{n-1}$ are distinct. 
Note that a simple cycle cannot be a simple path.

A digraph $G=(V,A)$ is said to be strongly connected if for any vertices $v,w \in V$ there exists a path from $v$ to $w$. 

\subsection{Boolean networks}

Suppose a Boolean dynamical system is defined by a map $f \colon \{0, 1\}^N \to \{0, 1\}^N$.
The Boolean network $\Sigma = (V, A^+, A^-)$ of $f$ is constructed in the following way;
\begin{enumerate}
\item $V = \{1, \dots, N \}$ is an $N$-point set.
\item $(V,A)$ is a digraph where $A = A^+ \sqcup A^-$.
\item $(i,j) \not \in A$ if $f_j$ is independent of $x_i$.
\item $(i,j) \in A^+$ if $f_j$ is increasing in $x_i$.
\item $(i,j) \in A^-$ if $f_j$ is decreasing in $x_i$.
\end{enumerate}  

An element $(i,j) \in A^+$ is called an activating edge, and is denoted by $ i \to j$.

An element $(i,j) \in A^-$ is called an inhibiting edge, and is denoted by $ i \dashv j$.

\begin{claim}
If $V = \{1, \dots, N \}$, $(V,A)$ is a digraph, and $A = A^+ \sqcup A^-$ is a partition of $A$, then
there exists a map $f \colon \{0, 1\}^N \to \{0, 1\}^N$ of which Boolean network is $(V, A^+, A^-)$.

On the other hand, even if $f$ and $f'$ have the same Boolean network $(V, A^+, A^-)$, $f$ may differ from $f'$.
\end{claim}

\begin{example}\label{example}
The following two distinct Boolean dynamical systems $f$ and $f'$ on $\{0,1\}^4$ have the same Boolean network $\Sigma$.

\[
\begin{cases}
f_1(x) = \bar{x_4} & {} \\
f_2(x) = x_1 & {} \\
f_3(x) = x_1 \vee \bar{x_2}& {} \\
f_4(x) = x_3 & 
\end{cases}
, \quad
\begin{cases}
f'_1(x) = \bar{x_4} & {} \\
f'_2(x) = x_1 & {} \\
f'_3(x) = x_1 \wedge \bar{x_2}& {} \\
f'_4(x) = x_3 & 
\end{cases}
\]

\end{example}

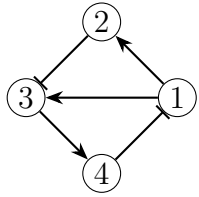
\begin{figure}[htbp]
\centering
\begin{tikzpicture}[
  vertex/.style={circle, draw, minimum size=5mm, inner sep=0pt},
  act/.style={-{Stealth}, thick},
  inh/.style={thick, -{Bar[width=7pt]}}
]
  \node[vertex] (v1) at (1,0) {$1$};
  \node[vertex] (v2) at (0,1) {$2$};
  \node[vertex] (v3) at (-1,0) {$3$};
  \node[vertex] (v4) at (0,-1) {$4$};

  \draw[act] (v1) -- (v2);
  \draw[act] (v1) -- (v3);

  \draw[inh] (v2) -- (v3);

  \draw[act] (v3) -- (v4);

  \draw[inh] (v4) -- (v1);

\end{tikzpicture}
\caption{The Boolean network $\Sigma$ in Example \ref{example}.}
\label{fig:example-network}
\end{figure}


Hereafter, graph-theoretic terminology for digraphs will also be used for Boolean networks. For example, $\Sigma$ is said to be strongly connected if its underlying digraph $G$ is strongly connected.

\section{The Strongly Connected Case}

\subsection{Characterization Theorem for Strongly Connected Case}

A Boolean dynamical system $f \colon \{0,1\}^N \to \{0,1\}^N$ is said to be oscillatory 
if it has no fixed point $x \in  \{0,1\}^N$.
A Boolean network $\Sigma$ is said to be structurally oscillatory 
if any Boolean dynamical system $f$ represented by $\Sigma$ is oscillatory.

Then, the main theorem in \cite{Az2} can be restated in the following way:

\begin{theorem}
Suppose that $\Sigma = (V, A^+, A^-)$ is a Boolean network.

\noindent (1) Suppose digraph $G = (V, A^+ \sqcup A^-)$ is cactus.
Then $\Sigma$ is structurally oscillatory if and only if any simple cycle of $G$ has an odd number of inhibiting edges.

\noindent (2) Suppose digraph $G = (V, A^+ \sqcup A^-)$ is cactus-expandable.
Then $\Sigma$ is structurally oscillatory if any simple cycle of $G$ has an odd number of inhibiting edges.
\end{theorem}

Definition of cactus and cactus-expandable digraphs can be found in \cite{Az1}, \cite{Az2}, and \cite{K}.
Here, we simply remark that they are stronger than being strongly connected.

The following is the main theorem of this section.

\begin{theorem}\label{main}
Suppose that $\Sigma = (V, A^+, A^-)$ is a Boolean network and the digraph $G = (V, A^+ \sqcup A^-)$ is strongly connected.
Then $\Sigma$ is structurally oscillatory if and only if any simple cycle of $G$ has an odd number of inhibiting edges.
\end{theorem}

\subsection{Proof of the Theorem}

Instead of proving the equivalence in Theorem \ref{main} directly, we prove the equivalence of the negations of its two statements.

\subsubsection{If part}

In this subsubsection, we shall prove the following proposition.

\begin{proposition}\label{if}
If a Boolean dynamical system $f \colon \{0,1\}^N \to \{0,1\}^N$ has a fixed point $\hat{x}\in\{0,1\}^N$,
the Boolean network of $f$ is $\Sigma=(V, A^+, A^-)$, and $\Sigma$ is  strongly connected,
then $\Sigma$ has a simple cycle with an even number of inhibiting edges.
\end{proposition}

\begin{definition}
Fix a Boolean dynamical system $f \colon \{0,1\}^N \to \{0,1\}^N$ and a fixed point $\hat{x}\in\{0,1\}^N$ of $f$.
Suppose Boolean network of $f$ is $\Sigma=(V, A^+, A^-)$.
We partition the edge set $A=A^+ \sqcup A^-$ into two colors.
Take an arc $(i,j)\in A$.

If $(i,j) \in A^+$ and $\hat{x}_i=\hat{x}_j$; or $(i,j) \in A^-$ and $\hat{x}_i\neq \hat{x}_j$, then $(i,j)$ is called an \textbf{admissible edge} and colored green.
If $(i,j) \in A^+$ and $\hat{x}_i\neq\hat{x}_j$; or $(i,j) \in A^-$ and $\hat{x}_i= \hat{x}_j$, then $(i,j)$ is called an \textbf{inadmissible edge} and colored red.


\begin{table}[htbp]
\centering
\begin{tabular}{cc}
\toprule
Admissible & Inadmissible \\
\midrule
$0 \to 0$ & $0 \to 1$ \\
$1 \to 1$ & $1 \to 0$ \\
$0 \dashv 1$ & $0 \dashv 0$ \\
$1 \dashv 0$ & $1 \dashv 1$ \\
\bottomrule
\end{tabular}
\caption{Admissible and inadmissible edges.}
\label{tab:admissible}
\end{table}

\end{definition}

\begin{lemma}
Fix a Boolean dynamics $f \colon \{0,1\}^N \to \{0,1\}^N$ and a fixed point $\hat{x}\in\{0,1\}^N$ of $f$.
If the Boolean network $\Sigma = (V,A^+,A^-)$ of $f$ is strongly connected, 
then each vertex $v_j \in V$ has at least one incoming admissible edge.

\end{lemma}

\begin{proof}
Since $\Sigma$ is strongly connected, $v_j$ has at least one incoming edge. Therefore, $f_j$ is not a constant function.

We prove the lemma by contradiction.
Without loss of generality, suppose $\hat{x}_j=1$.
We assume that all incoming edges to $v_j$ are inadmissible.
Let 
$I^+:=\{i \mid \text{$(i,j)$ is an activating edge}\}$ 
and
$I^-:=\{i \mid \text{$(i,j)$ is an inhibiting edge}\}$.

Since $\hat{x}$ is a fixed point of $f$, $f_j(\hat{x})=\hat{x}_j=1$.
By definition, 
$\hat{x}_i=0$ if $i\in I^+$ and 
$\hat{x}_i=1$ if $i\in I^-$.
On the other hand, the function $f_j(x_1,\dots,x_N)$ is 
increasing in $x_i$ if $i\in I^+$,
decreasing in $x_i$ if $i\in I^-$, and
independent of $x_i$ if $i \not\in I^+ \sqcup I^-$.
Therefore, $f_j(x_1,\dots,x_N)=1$ for any $(x_1,\dots,x_N)$, so it is a constant function; this is a contradiction.

We can treat the case $\hat{x}_j=0 $ in a similar way.
\end{proof}

\begin{proof}[Proof of the proposition \ref{if}]

Take an arbitrary node $v_j \in V$ and set $j_0=j$.
$v_{j_0}$ has at least one incoming admissible edge $(j_1,j_0)$. 
Inductively, suppose $(j_{i+1},j_i)$ is an admissible edge for $i=1,2,\dots$.
Since $V$ is a finite set, there exist $n_0$ and $n_1$ so that
$n_0 < n_1$, $j_{n_0} = j_{n_1}$, and
$(j_{n_1}, j_{n_1-1}, \dots, j_{n_0+1}, j_{n_0})$ is a simple cycle of $\Sigma$ consisting of admissible edges.
Since this cycle is composed of admissible edges, the number of inhibiting edges should be even because of the parity.

\end{proof}

\subsubsection{Only If part}

In this subsubsection, we shall prove the other way.

\begin{proposition}\label{onlyif}
If $\Sigma = (V, A^+, A^-)$ is strongly connected and has a simple cycle with an even number of inhibiting edges,
then there exists a map $f \colon \{0,1\}^N \to \{0,1\}^N$ with a fixed point $\hat{x}\in\{0,1\}^N$,
such that the Boolean network of $f$ is $\Sigma$.
\end{proposition}


Before proving the proposition, we will define some notation of a digraph.
The notion of out-trees (or arborescences) is defined by Deo \cite[pp206--207, section 9-6]{D}.

\begin{definition}
A digraph $T$ is said to be an \textbf{out-tree} (or an arborescence) if
\begin{enumerate}
\item $T$ is a tree as a non-directed graph.
\item $T$ has a unique vertex $r$ of zero in-degree.
\end{enumerate}
This node $r$ is called the root of the out-tree. An example of an out-tree is shown in Figure \ref{fig:out-tree}.
\end{definition}

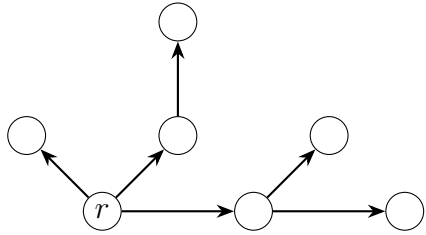
\begin{figure}[htbp]
\centering
\begin{tikzpicture}[
    >=Stealth,
    vertex/.style={
        circle,
        draw,
        minimum size=5mm,
        inner sep=0pt
    },
    edge/.style={->, thick}
]

\node[vertex] (r) at (0,0) {$r$};

\node[vertex] (v1) at (-1,1) {};
\node[vertex] (v2) at (1,1) {};
\node[vertex] (v3) at (1,2.5) {};
\node[vertex] (v4) at (2,0) {};
\node[vertex] (v5) at (3,1) {};
\node[vertex] (v6) at (4,0) {};

\draw[edge] (r) -- (v1);
\draw[edge] (r) -- (v2);
\draw[edge] (v2) -- (v3);
\draw[edge] (r) -- (v4);
\draw[edge] (v4) -- (v5);
\draw[edge] (v4) -- (v6);
\end{tikzpicture}
\caption{An out-tree $T$ with a root $r$.}
\label{fig:out-tree}
\end{figure}

\begin{definition}
Suppose that $C$, $T_1$, \dots, $T_k$ are subdigraphs of a digraph $G$ so that
$C$ is a simple cycle and $T_i$ is an out-tree with a root $r_i$.
We also assume that $C \cap T_i = \{r_i\}$ and $T_i \cap T_j = \emptyset$ for $i \neq j$.
In such a case, we call a subdigraph $B := C \cup T_1 \cup \cdots \cup T_k$ a \textbf{bonsai} on $C$.
\end{definition}

\begin{figure}[htbp]
\centering
\begin{tikzpicture}[
    >=Stealth,
    vertex/.style={
        circle,
        draw,
        minimum size=5mm,
        inner sep=0pt
    },
    edge/.style={->, thick}
]

\node[vertex] (v1) at (0,2) {$r_1$};
\node[vertex] (v2) at (2,2) {$r_2$};
\node[vertex] (v3) at (2,0) {$r_3$};
\node[vertex] (v4) at (0,0) {};

\draw[edge] (v1) -- (v4);
\draw[edge] (v4) -- (v3);
\draw[edge] (v3) -- (v2);
\draw[edge] (v2) -- (v1);

\node[vertex] (a1) at (-1.2,3.2) {};
\node[vertex] (a2) at (-2.2,4.2) {};
\node[vertex] (a3) at (-0.2,4.2) {};

\draw[edge] (v1) -- (a1);
\draw[edge] (a1) -- (a2);
\draw[edge] (a1) -- (a3);

\node[vertex] (b1) at (3.2,3.0) {};
\node[vertex] (b2) at (4.2,4.0) {};

\draw[edge] (v2) -- (b1);
\draw[edge] (b1) -- (b2);

\node[vertex] (c1) at (3.5,-1.0) {};
\node[vertex] (c2) at (4.5,-2.0) {};
\node[vertex] (c3) at (2.5,-1.5) {};

\draw[edge] (v3) -- (c1);
\draw[edge] (c1) -- (c2);
\draw[edge] (v3) -- (c3);

%

\node at (1,1) {$C$};
\node at (-1.8,5.0) {$T_1$};
\node at (4.5,4.6) {$T_2$};
\node at (3.5,-2.0) {$T_3$};

\end{tikzpicture}
\caption{A bonsai on a simple cycle $C$.}
\label{fig:bonsai}
\end{figure}
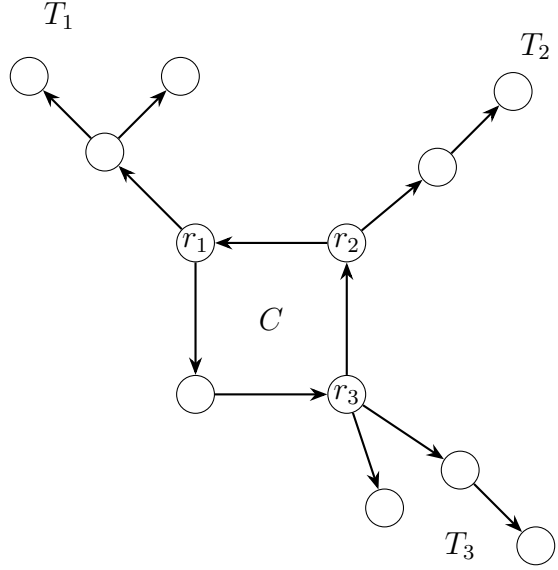

See Figure \ref{fig:bonsai} for an example of bonsai.
We note that every vertex in a bonsai has an incoming edge of a bonsai.
Here is an important lemma for a bonsai.

\begin{lemma} (Existence of a Spanning Bonsai)
Suppose $C$ is a simple cycle in a strongly connected digraph $G$.
Then there exists a bonsai $B$ on $C$ that includes all the vertices of $G$.
Such a bonsai is called a \textbf{spanning bonsai} in $G$.
\end{lemma}

\begin{proof}
Suppose $\mathcal{B}(G,C)$ is the set of all bonsais on $C$ in $G$.
Since $\mathcal{B}(G,C)$ is a finite set, we can take a maximal element $B_{\mathrm{max}}$ among
$\mathcal{B}(G,C)$ with respect to inclusion. Then $B_{\mathrm{max}}$ is a spanning bonsai.
Otherwise, take vertices $w \not \in B_{\mathrm{max}}$ and $v  \in B_{\mathrm{max}}$.
Since $G$ is strongly connected, there exists a simple path $(v=v_0, v_1,\dots,v_n=w)$ from $v$ to $w$.
Set $i:= \max\{j \mid v_j \in B_{\mathrm{max}}\}$. 
Then $B_{\mathrm{max}} \cup (v_i,\dots,v_n)$ is also a bonsai on $C$, 
since $v_{i+1},\dots,v_n \not\in B_{\mathrm{max}}$, attaching the tail of the simple path creates no additional cycle.
That contradicts the fact that $B_{\mathrm{max}}$ is maximal.
\end{proof}

Now we can prove our proposition.

\begin{proof}[Proof of the proposition \ref{onlyif}]

Suppose $\Sigma = (V, A^+, A^-)$ is strongly connected and $C$ is a simple cycle with an even number of inhibiting edges.
Then, from the lemma above, there is a spanning bonsai $B \subset \mathcal{B}(\Sigma,C)$.

We set a candidate for a fixed point $\hat{x}$ so that every edge of $B$ is admissible with respect to $\hat{x}$.
This is possible since $B$ has only one cycle $C$, which has an even number of inhibiting edges.

Then we shall construct $f$ so that $f(\hat{x})=\hat{x}$ and $f$ is compatible with $\Sigma=(V, A^+, A^-)$.
This can be done quite simply. The main idea is that if $\hat{x}_j=0$ then construct $f_j$ by \textit{AND},
and if $\hat{x}_j=1$ then construct $f_j$ by \textit{OR}.

To be precise, if $\hat{x}_j=0$ then
\[
f_j(x) = \left( \bigwedge_{(i,j)\in A^+} x_i \right) \wedge \left( \bigwedge_{(i,j)\in A^-} \bar{x_i} \right),
\]
and if $\hat{x}_j=1$ then
\[
f_j(x) := \left( \bigvee_{(i,j)\in A^+} x_i \right) \vee \left( \bigvee_{(i,j)\in A^-} \bar{x_i} \right).
\]
Then $f(\hat{x})=\hat{x}$ and $f$ is compatible with $\Sigma=(V, A^+, A^-)$.
Since $B$ is a spanning bonsai, every vertex has at least one incoming admissible edge in $B$.
Hence each coordinate satisfies $f_j(\hat{x})=\hat{x}_j$.
Therefore, this $f$ fixes $\hat{x}$.
\end{proof}

Theorem \ref{main} follows from propositions \ref{if} and \ref{onlyif}.

%


\section{General case}

Now we are going to study the case where $\Sigma$ is not strongly connected.
Note that if $j \in V$ has no incoming edge, then $f_j$ does not depend on anything, so it is a constant function.


The main idea is the following. See Figure \ref{fig:networks}. (a) is structurally oscillatory because it is strongly connected and its unique simple cycle has
an odd number of inhibiting edges.
(b) is not structurally oscillatory  because it is strongly connected and has a simple cycle with an even number of inhibiting edges.
(c) is structurally oscillatory, because the outgoing edge does not affect the dynamical system.
However, (d) is not structurally oscillatory, because $(0,0,0,0,0,0)$ can be a fixed point. 
This means that inward edges do affect the dynamical system.
So the idea is that we should focus on the most upstream region.

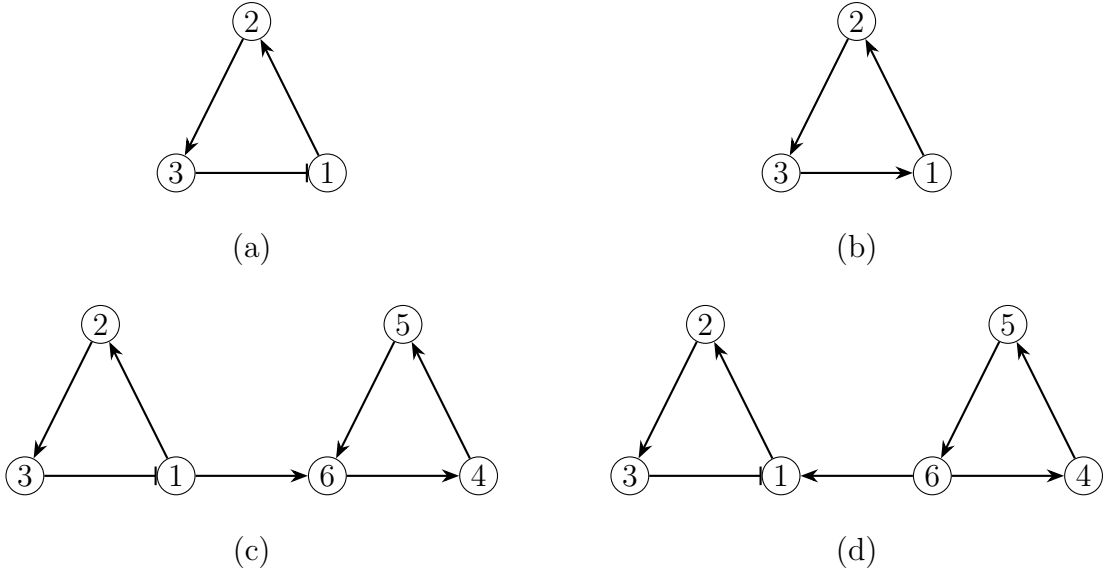
\begin{figure}[htbp]
\centering
\begin{tikzpicture}[
  vertex/.style={circle, draw, minimum size=5mm, inner sep=0pt},
  act/.style={-{Stealth}, thick},
  inh/.style={thick, -{Bar[width=7pt]}}
]
  \node[vertex] (a1) at (-3,4) {$1$};
  \node[vertex] (a2) at (-4,6) {$2$};
  \node[vertex] (a3) at (-5,4) {$3$};

  \node[vertex] (b1) at (5,4) {$1$};
  \node[vertex] (b2) at (4,6) {$2$};
  \node[vertex] (b3) at (3,4) {$3$};

  \node[vertex] (c1) at (-5,0) {$1$};
  \node[vertex] (c2) at (-6,2) {$2$};
  \node[vertex] (c3) at (-7,0) {$3$};
  \node[vertex] (c4) at (-1,0) {$4$};
  \node[vertex] (c5) at (-2,2) {$5$};
  \node[vertex] (c6) at (-3,0) {$6$};

  \node[vertex] (d1) at (3,0) {$1$};
  \node[vertex] (d2) at (2,2) {$2$};
  \node[vertex] (d3) at (1,0) {$3$};
  \node[vertex] (d4) at (7,0) {$4$};
  \node[vertex] (d5) at (6,2) {$5$};
  \node[vertex] (d6) at (5,0) {$6$};

  \draw[act] (a1) -- (a2);
  \draw[act] (a2) -- (a3);
  \draw[inh] (a3) -- (a1);

  \draw[act] (b1) -- (b2);
  \draw[act] (b2) -- (b3);
  \draw[act] (b3) -- (b1);

  \draw[act] (c1) -- (c2);
  \draw[act] (c2) -- (c3);
  \draw[inh] (c3) -- (c1);
  \draw[act] (c4) -- (c5);
  \draw[act] (c5) -- (c6);
  \draw[act] (c6) -- (c4);
  \draw[act] (c1) -- (c6);

  \draw[act] (d1) -- (d2);
  \draw[act] (d2) -- (d3);
  \draw[inh] (d3) -- (d1);
  \draw[act] (d4) -- (d5);
  \draw[act] (d5) -- (d6);
  \draw[act] (d6) -- (d4);
  \draw[act] (d6) -- (d1);

\node at (-4,3) {(a)};
\node at (4,3) {(b)};
\node at (-4,-1) {(c)};
\node at (4,-1) {(d)};

\end{tikzpicture}
\caption{Examples illustrating structural oscillatority.}
\label{fig:networks}
\end{figure}

\subsection{Preorder and Strongly Connected Component}

We define a binary relation $ \precsim $ on $V$ by
$x \precsim y $ if there exist a directed path from $x$ to $y$ or $x=y$. 
Then, $ \precsim $ is a preorder; that is, it satisfies:
\begin{enumerate}
\item Reflexivity: $ x \precsim x$ for all $ x \in V$.
\item Transitivity: If $ x \precsim y$ and $ y \precsim z$ then $ x \precsim z$ for all $x,y,z \in V$.
\end{enumerate}

We write $x \sim y$ if $x \precsim y$ and $y \precsim x$. Then $\sim$ is an equivalence relation.
Suppose $V/\sim \, = \{V_1, \dots, V_l\}$.
Take $\Sigma_i=(V_i,A_i^+,A_i^-)$ and $G_i=(V_i,A_i)$ as Boolean full subnetworks and full subdigraphs; that is,
$$
A_i^+ := \{ (x,y) \mid \text{$x,y \in V_i$ and $(x,y)\in A^+$} \},
$$
$$
A_i^- := \{ (x,y) \mid \text{$x,y \in V_i$ and $(x,y)\in A^-$} \},
$$
$$
A_i := \{ (x,y) \mid \text{$x,y \in V_i$ and $(x,y)\in A$} \} = A_i^+ \sqcup A_i^-.
$$

Each full subdigraph $G_i$ is called a strongly connected component of $G$. $G_i$ is either a single vertex without a loop,
or a maximal strongly connected subdigraph of $G$.

The preorder $ \precsim $ on $V$ induces a partial order $\leq$ on $V/\sim$.
Note that $V_i \leq V_j$ if and only if for any $x \in V_i$ and $y \in V_j$ there exists a directed path from $x$ to $y$, or $V_i=V_j$ is a single vertex without a loop.

We also induce the same order on $\{\Sigma_1, \dots, \Sigma_l \}$ and $\{G_1, \dots, G_l\}$, that is,
$V_i \leq V_j \Leftrightarrow \Sigma_i \leq \Sigma_j \Leftrightarrow G_i \leq G_j$.

Define $\mathcal{S}_{\mathrm{min}}$ to be the set of all minimal elements of $(\{\Sigma_1, \dots, \Sigma_l \}, \leq)$.
Now we can state the characterization theorem for the general case:

\begin{theorem}\label{general}
$\Sigma$ is structurally oscillatory if and only if 
at least one of $\Sigma_i \in \mathcal{S}_{\mathrm{min}}$ is structurally oscillatory.
\end{theorem}

Since $\Sigma_i$ is either strongly connected or a single vertex, Theorems \ref{main} and \ref{general}
give the perfect criterion for structurally oscillatory.

\begin{proof}[Proof of theorem \ref{general}]

If $\Sigma_i \in \mathcal{S}_{\mathrm{min}}$ is structurally oscillatory, 
then $f \mid _{ \{0,1\}^{V_i}} $ cannot have a fixed point,
so as $f$.

Suppose that all $\Sigma_i \in \mathcal{S}_{\mathrm{min}}$ are not structurally oscillatory.
Then, each $\Sigma_i$ either is a single vertex $\{v_i\}$ without a loop, or has a bonsai $B_i$ on a simple cycle $C_i$ with an even number of  
inhibiting edges.

Then consider a set $\mathcal{F}=\{(\check{B}_1,\dots,\check{B}_l)\}$, where 
$\check {B_i}$ is an out-tree rooted on $v_i$ if $\Sigma_i=\{v_i\}$,
otherwise $\check {B_i}$ is a bonsai on $C_i$ and $\check {B_i} \supset B_i$, and $\check{B}_1,\dots,\check{B}_l$ are disjoint.%
\footnote{
The notation $\check{B}_i$ is intended to suggest that $B_i$ is extended by attaching branches.
}

We define the inclusion relation on $\mathcal{F}$ by $(\check{B}_1,\dots,\check{B}_l) \subset (\check{B}'_1,\dots,\check{B}'_l)$
if and only if $\check{B}_i \subset \check{B}'_i$ for every $i\in\{1,\dots,l\}$.

\begin{claim}
Then the maximal element $F_\mathrm{max}$ among $\mathcal{F}$, in a sense of inclusion, covers every vertex of $\Sigma$.
\end{claim}

\begin{proof}[Proof of the claim]
The idea of the proof is the same as one of the existence of a spanning bonsai.

Suppose $F_\mathrm{max} = (\check{B}_1,\dots,\check{B}_l)$ is a maximal element among $\mathcal{F}$ and 
$w$ is a vertex of $\Sigma$ not covered by $F_\mathrm{max}$ that belongs to a strongly connected component $\Sigma_j$. 
If $\Sigma_j \in \mathcal{S}_{\mathrm{min}}$ then $w$ is covered by $B_j$, which is a contradiction. Therefore $\Sigma_j \not\in \mathcal{S}_{\mathrm{min}}$. Since there are only finite numbers of connected components,
there exists $\Sigma_i \in \mathcal{S}_{\mathrm{min}}$ so that $\Sigma_i < \Sigma_j$.
Take a vertex $v$ from $\Sigma_i$, then $v\neq w$ and $v \precsim w$, therefore there is a path $(v=v_0, v_1,\dots, v_n=w)$.
Then set
\[
m := \max\{ k \in \{0,1,\dots,n\} \mid \text{$v_k \in \check{B}_{i'}$ for some $i' \in \{1,2,\dots,l\}$ } \}.
\]
Suppose that $v_m \in \check{B}_{i'}$, where $i'$ may differ from $i$.
Anyway, set $\check{B}'_{i'} := \check{B}_{i'} \cup (v_m,\dots,v_n)$.
Taking $\check{B'}_{i'}$ instead of $\check{B}_{i'}$, $F=(\check{B}_1,\dots,\check{B'}_{i'},\dots,\check{B}_l)$ is still a member of 
$\mathcal{F}$, because $v_{m+1},\dots,v_n$ does not belong to $\check{B}_{s}$ for any $s \in \{1,\dots,l\}$, 
the path $(v_m,\dots,v_n)$ is disjoint from $\check{B}_{s}$ for any $s \in \{1,\dots,l\}\setminus \{i'\}$ and
creates no simple cycle in $\check'{B}_{i'}$.
This contradicts the assumption that $F_\mathrm{max} = (\check{B}_1,\dots,\check{B}_l)$ is a maximal element among $\mathcal{F}$.
\end{proof}

$F_\mathrm{max}$ covers every vertex and each of its connected components is either an out-tree or a bonsai on a simple cycle with even number of inhibiting edges.
Therefore we can create a fixed point $\hat{x}$ and a map $f$ in the same way as Proof of the proposition \ref{onlyif}.
\end{proof}

\section*{Acknowledgement}
I would like to express my gratitude for the meaningful discussions with Shun-ichi Azuma. Additionally, 
I appreciate Ryosuke Iritani for his contributions and support.

\end{document}